\documentclass[a4paper,reqno, 12pt]{amsart}

\usepackage[parfill]{parskip}    
\usepackage{amsrefs, amsmath}
\usepackage{amsfonts, mathrsfs}
\usepackage{amssymb}
\usepackage{amscd}
\usepackage{fullpage, color}
\usepackage[labelformat=empty]{subfig}
\usepackage{booktabs}
\usepackage{xypic}
\usepackage[dvips]{graphicx,epsfig}
\usepackage{enumerate}

\BibSpec{article}{%
+{}{\sc \PrintAuthors} {author}
+{,}{ \textrm} {title}
+{,}{ \textit} {journal}
+{,}{ } {volume}
+{}{ \IfEmptyBibField{journal}{}{\PrintDate{date}}} {transition}
+{,}{ } {pages}
+{,}{ } {status}
+{.}{} {transition}
+{}{ Preprint\IfEmptyBibField{date}{}{ \PrintDate{date}}, available at \tt} {eprint}
+{.}{} {transition}
}

\BibSpec{book}{%
+{}{\sc \PrintAuthors} {author}
+{,}{ \textit} {title}
+{.}{ } {series}
+{,}{ } {volume}
+{.}{ } {publisher}
+{,}{ } {place}
+{,}{ } {date}
+{.}{} {transition}
}

\BibSpec{collection.article}{
+{} {\sc \PrintAuthors} {author}
+{,} { \textrm } {title}
+{,} { in \it \PrintConference} {conference}
+{,} { pp.~} {pages}
+{.} {\PrintBook} {book}
+{,} { \PrintDateB} {date}
+{.} {} {transition}
}

\BibSpec{innerbook}{
+{.} { \emph} {title}
+{.} { } {part}
+{:} { \emph} {subtitle}
+{.} { } {series}
+{,} { } {volume}
+{.} { Edited by \PrintNameList} {editor}
+{.} { Translated by \PrintNameList}{translator}
+{.} { \PrintContributions} {contribution}
+{.} { } {publisher}
+{.} { } {organization}
+{,} { } {address}
+{,} { \PrintEdition} {edition}
+{,} { \PrintDateB} {date}
+{.} { } {note}
}

\allowdisplaybreaks[3]

\DeclareMathOperator{\tr}{Tr}
\DeclareMathOperator{\diag}{diag}
\newcommand{\de}{{\partial}}

\newcommand{\rd}{\mathrm{d}}
\newcommand{\ri}{\mathrm{i}}
\newcommand{\re}{\mathrm{e}}

\newcommand{\bbN}{\mathbb{N}}

\newcommand{\bbZ}{\mathbb{Z}}

\newcommand{\bbC}{\mathbb{C}}
\newcommand{\bbP}{\mathbb{P}}

\def\bary{\begin{array}} 
\def\eary{\end{array}} 
\def\ben{\begin{enumerate}} 
\def\een{\end{enumerate}}
\def\bit{\begin{itemize}} 
\def\eit{\end{itemize}}
\def\nn{\nonumber} 

\newcommand{\cO}{\mathcal{O}}

\newcommand{\LL}{\mathcal{L}}

\newcommand{\cN}{\mathcal{N}}

\newcommand{\HH}{\mathcal{H}}

\newcommand{\cF}{\mathcal{F}}

\newcommand{\cM}{\mathcal M}

\newcommand{\ii}{\mathrm{i}}
\newcommand{\di}{\mathrm{d}}

\def\L{\Lambda}

\def\beq{\begin{equation}}                     %
\def\eeq{\end{equation}}                       %
\def\bea{\begin{eqnarray}}                     
\def\eea{\end{eqnarray}}
\def\bary{\begin{array}} 
\def\eary{\end{array}} 
\def\ben{\begin{enumerate}} 
\def\een{\end{enumerate}}
\def\bit{\begin{itemize}} 
\def\eit{\end{itemize}}
\def\nn{\nonumber} 
\def\de {\partial}

\def\a{\alpha}
\def\b{\beta}
\def\g{\gamma}
\def\d{\delta}
\def\e{\epsilon}

\def\Res{\mathrm{Res}}

\theoremstyle{plain}

\newtheorem{thm}{Theorem}[section]
\newtheorem{lem}[thm]{Lemma}

\newtheorem{prop}[thm]{Proposition}
\newtheorem{conj}[thm]{Conjecture}

\newtheorem*{conj*}{Conjecture}

\newtheorem*{cor*}{Corollary}

\newtheorem{defn}{Definition}[section]

\theoremstyle{definition}

\newtheorem{rem}[thm]{Remark}
\newtheorem{rmk}[thm]{Remark}

\newcommand{\Li}{\operatorname{Li}}

\newcommand{\GIT}[1]{/\!\!/_{\kern-.2em #1 \kern0.1em}}

\renewcommand{\l}{\left}
\renewcommand{\r}{\right}
\newcommand{\bra}{\left\langle}
\newcommand{\ket}{\right\rangle}
\newcommand{\pf}{\noindent {\it Proof.} }

\newcommand{\ev}{\operatorname{ev}}

\def\bred{\begin{color}{red}}
\def\ered{\end{color}}

\def\bes{\begin{subequations}}
\def\ees{\end{subequations}}

\begin{document}

\title[Integrable hierarchies and the mirror model of local $\bbC
  \bbP^1$]{Integrable hierarchies and the mirror model of local $\bbC \bbP^1$}

\author{Andrea Brini}
\address{Section de Math\'ematiques et D\'epartement de Physique Th\'eorique,
  Universit\'e de Gen\`eve, 24 quai Ansermet, CH-1211, Geneva, Switzerland}
\email{Andrea.Brini@unige.ch}
\author{Guido Carlet}
\address{Dipartimento di Matematica e Applicazioni, Universit\`a di Milano - Bicocca, via R. Cozzi, 53, 20125 Milan, Italy}
\email{guido.carlet@unimib.it}
\author{Paolo Rossi}
\address{Institut de Math\'ematiques de Jussieu, Universit\'e Pierre et Marie Curie - Paris 6, 4 place de Jussieu, 75005 Paris, France}
\email{issoroloap@gmail.com}
\subjclass[2000]{81T45 (primary), 81T30, 57M27, 17B37, 14N35}
\keywords{Gromov-Witten, integrable hierarchies, mirror symmetry, 2D-Toda, Ablowitz-Ladik.}
   \dedicatory{Dedicated to Boris Anatolevich Dubrovin on the occasion
     of his $60^{\rm th}$ birthday, with friendship and gratitude}

\begin{abstract}
We study structural aspects of the Ablowitz-Ladik (AL) hierarchy in the light of its
realization as a two-component reduction of the two-dimensional Toda
hierarchy, and establish new results on its connection to the
Gromov-Witten theory of local $\bbC\bbP^1$. We first of all elaborate on the
relation to the Toeplitz lattice 
and obtain a neat
description of the Lax formulation of the AL system. We then study
the dispersionless limit and rephrase it in terms of a conformal semisimple
Frobenius manifold with non-constant unit, whose properties we thoroughly
analyze. We build on this connection along two main strands. First of all, we
exhibit a manifestly local bi-Hamiltonian structure of the Ablowitz-Ladik
system in the zero-dispersion limit. Secondarily, we make precise the relation between this canonical Frobenius
structure and the one that underlies the Gromov-Witten theory of the resolved
conifold 
in the equivariantly Calabi-Yau case; a key role
is played by Dubrovin's notion of ``almost duality'' of Frobenius manifolds. As a
consequence, we obtain a derivation of genus zero mirror symmetry for
local $\bbC\bbP^1$ in terms of a dual logarithmic Landau-Ginzburg model.
\end{abstract}

\maketitle
\tableofcontents

\clearpage
\section{Introduction}
Integrable hierarchies find a special place of appearance in moduli
space problems motivated by topological field theory. A prominent case study is provided by the classical integrable hierarchies that
conjecturally govern the Gromov--Witten theory of symplectic manifolds. 
Denote with $\overline{\cM}_{g,n}(X,\beta)$ the stable compactification \cite{Kontsevich:1994na} of the
moduli space of degree $\beta\in H_2(X,\bbZ)$ $J$-holomorphic maps from $n$-pointed,
arithmetic genus $g$ curves to a K\"ahler manifold $(X,J,\omega)$. The Gromov--Witten invariants of $X$ are defined as
\beq
\bra \tau_{p_1}(\phi_{\a_1}) \dots  \tau_{p_n}(\phi_{\a_n}) \ket_{g,n,\b}^X :=
\int_{[\overline{\cM}_{g,n}(X,\b)]^{\rm vir}} \prod_{i=1}^n \ev^*_i(\phi_{\a_i})\psi_i^{p_i}
\eeq
where $[\overline{\cM}_{g,n}(X,\b)]^{\rm vir}$ is the virtual fundamental
class of $\overline{\cM}_{g,n}(X,\b)$, 
$\phi_{\a_i}\in
H^\bullet(X, \bbC)$ are arbitrary co-homology classes of $X$,
$\ev_i:\overline{\cM}_{g,n}(X,\b)\to X$ is the evaluation map at the $i^{\rm th}$
marked point, and $\psi_i=c_1(\LL_i)$ are the first Chern classes of the
universal cotangent line bundles $\LL_i$ on $\overline{\cM}_{g,n}(X,\b)$.
These numbers are interesting from a variety of points of view: in string
theory, they compute
worldsheet instanton effects for type IIA strings; in symplectic topology, they yield a highly sophisticated set of
invariants of the symplectic structure $\omega$; in enumerative algebraic
geometry, they
have an interpretation as a ``virtual count'' of holomorphic
curves inside $X$. \\

Kontsevich's celebrated proof \cite{Kontsevich:1992ti} of  Witten's conjecture \cite{MR1068086, Witten:1990hr}
 relating the Korteweg--de Vries hierarchy to 
intersection theory on the
Deligne--Mumford moduli space of curves $\overline{\cM}_{g,n}$ suggested a
connection between Gromov--Witten theory and integrable systems in the
following form. Let $\e$ and $t^{\alpha,p}$ be formal symbols, where $\a =\{1,\dots, h_X\}$, $h_X
  :=\dim_\bbC H^\bullet(X, \bbC)$ and $p \in \bbN$, and denote with
$\mathbf{t}$ the set $\mathbf{t}:=\l\{t^{\a,p}\r\}_{\substack{\a \in h_X
    \\ p\in \bbN}}$. Write $\phi_1$ for the identity of $H^\bullet(X)$ and define $x:=t^{1,0}$.
The all-genus, full-descendant Gromov--Witten potential of $X$ is the formal power series
\bea
 \mathcal{F}^X(\e, \mathbf{t}) = \sum_{g \geq 0} \e^{2g-2}\sum_{ \beta \in H_2(X,\bbZ)}\sum_{n \geq 0} \sum_{p_1,\ldots,p_n}
   {\prod_{i=1}^n t^{\a_i,p_i} \over n!}
\bra \tau_{p_1}(\phi_{\a_1}) \dots  \tau_{p_n}(\phi_{\a_n}) \ket_{g,n,\b}^X.
\label{eq:potX}
\eea
We then have the following
\begin{conj}
\label{conj:integrableGW}
Let $\cF^X(\e, \mathbf{t})$ denote the all-genus full descendant Gromov-Witten potential of $X$. Then there exists a Hamiltonian integrable hierarchy of PDEs such that  $\e^2 \mathcal{F}^X(\e, \mathbf{t})$ is the logarithm of a $\tau$--function associated with one of its solutions. The variables $t^{\a,p}$ are identified with times of the hierarchy, and the genus counting variable $\e$ with a perturbative parameter in a small dispersion expansion of the equations.
\end{conj}
The case $X=\{\mathrm{pt}\}$ is the statement of the Witten--Kontsevich theorem. This connection has provided a mutually fruitful source of insights for the integrable
systems community on one hand, and for symplectic and algebraic geometers on
the other. \\

After the Witten--Kontsevich theorem, a lot of effort has been put to further elucidate the origin of integrability
in Gromov--Witten theory, and to find constructuve proofs of Conjecture \ref{conj:integrableGW}
for more general target spaces. Research in this direction has
received much attention in the early 90's, starting from the discovery
by Dubrovin and Krichever of a clear link
between topological Landau--Ginzburg models and the hydrodynamics of weakly
deformed soliton lattices \cite{Krichever:1992qe, Dubrovin:1992eu}.
Subsequently, the field gained further momentum from Dubrovin's systematic
study of WDVV equations \cite{MR1068086, Dijkgraaf:1990qw} and his universal
construction \cite{Dubrovin:1992dz}, for  arbitrary homogeneous
chiral algebras, of dispersionless (bi-)Hamiltonian integrable
hierarchies that encode the descendent sector of the theory - namely, in the case of quantum co-homologies, the complete set of descendent genus zero
Gromov--Witten invariants. \\ In more recent times, the (hard) task of
incorporating dispersive corrections in the picture - corresponding, in the
original Witten-Kontsevich picture, to higher genus Gromov--Witten invariants
- has followed two main strands. On one hand, the Dubrovin--Zhang
program of classification of normal forms of bi-Hamiltonian evolutionary
hierarchies has provided a concrete incarnation of Virasoro constraints from the
integrable system point of view, along with a complete reconstruction theorem
for higher genus descendent invariants \cite{dubrovin-2001}; on the other, generalizations of the Witten--Kontsevich
correspondence were explicitly constructed for the quantum co-homology of
simple target orbifolds, such as $BG$ \cite{MR1950944} and $[\mathbb{CP}^1/G]$
\cite{MR2199226, MR2433616, pauljohnsonthesis, rossi-2008}.
\subsection{The resolved conifold and the Ablowitz--Ladik hierarchy}
Inspired by the appearance of trilogarithmic prepotentials in Dubrovin's study
of the Ablowitz--Ladik (AL) hierarchy \cite{MR2462355}, one of us proposed in
\cite{Brini:2010ap} that Conjecture \ref{conj:integrableGW} should
hold true when $X$ is a particular local Calabi-Yau manifold of
dimension three -- the resolved conifold -- 
given by the strict transform 
of the
nodal quadric in $\mathbb{A}^4$, and the corresponding integrable hierarchy is
the Ablowitz--Ladik hierarchy \cite{MR0377223}. The precise statement, which was proven in
\cite{Brini:2010ap} at the first few orders in the genus expansion, relates a
peculiar form of the AL hierarchy to
the Gromov-Witten theory of $\cO_{\bbP^1}(-1)_{[\nu]}\oplus\cO_{\bbP^1}(-1)_{[-\nu]}$, equivariant with respect to a
fiber-wise $T\simeq \bbC^*$-action which covers the trivial action on the base
$\bbP^1$ and rotates the fibers with opposite weights; we denoted with $\nu$ the first Chern class of the
line bundle $\cO(1)\to BT \simeq \bbC\bbP^\infty$. Restricting to genus zero, primary
invariants, this statement can be rephrased as the equality of the quantum
co-homology ring of the resolved conifold in the equivariantly Calabi-Yau
case with the Frobenius structure that arises from a particular solution of
WDVV, first encountered in the treatment of the AL system in
\cite{MR2462355}. \\

This example, which is of remarkable importance in the Gromov--Witten theory
of Calabi-Yau threefolds, raised various interesting questions: among them,
the possibility to provide a local mirror symmetry construction for this
equivariant case, and the explanation of the apparent breakdown of bi-Hamiltonianity on the integrable
system side. In this paper we study both aspects in detail, by regarding the
AL hierarchy as the Toeplitz reduction of 2D-Toda \cites{MR1794352, MR1993935,
  MR2534519}. 
We first of all build, in Section~\ref{sec:hom}, on the identification of the AL lattice with the Toeplitz lattice and provide a clean proof of the invariance of the Toeplitz condition under the 2D-Toda flows at the dispersive level in the bi-infinite case. By observing that the Toeplitz Lax matrices admit a factorization in terms of two bi-diagonal matrices, we show that the Toeplitz lattice is an instance of rational reduction of the 2D-Toda hierarchy, which can be defined in general by analogy with the rational reductions of KP hierarchy~\cite{MR1340299, MR1352392}.
We then apply the dispersionless Lax formalism of 2D-Toda to
associate a new, canonical Frobenius structure with the hierarchy. This new
Frobenius manifold is {\it different} from the one that appears in
Gromov--Witten theory: it fails to have a covariantly constant unit vector
field, but it satisfies {\it all} the other axioms of a Frobenius
manifold, including somewhat surprisingly the existence of a linear Euler vector field. In
particular, this entails the existence of a local bi-Hamiltonian structure of
Dubrovin--Novikov type \cite{MR1037010}, to be contrasted with the
inhomogeneity of the prepotentials in \cite{MR2462355, Brini:2010ap} and the non--locality of the
pairs constructed in \cite{MR2238744}. \\
A natural question that arises is then how this new Frobenius structure and
the quantum co-homology of the resolved conifold are
related to one another. We find  in Section \ref{sec:inhom} that the relation
in question is remarkably given in the form of Dubrovin's almost
duality of Frobenius manifolds \cite{MR2070050}. By pushing the dispersionless
Lax formalism through the duality we obtain a logarithmic Landau--Ginzburg
mirror for local $\mathbb{CP}^1$, close in form to the LG models proposed by Hori, Iqbal and
Vafa for the non-equivariant theory \cite{Hori:2000ck}. We finally study the period
structure of the resulting
almost Frobenius manifold, and find a remarkable connection to the theory of singularities of divisors
considered in the context of twisted Picard--Lefschetz theory by Givental in
\cite{MR936695}. We conclude in Section \ref{sec:concl} with remarks on open
problems and new avenues of research.

\vspace{1cm}

\subsection*{Acknowledgements.} It is a privilege to dedicate this work to
Boris Dubrovin, for his guidance during our PhD years in Trieste and
for the many enlightening (and fun!) discussions we have had with him over the last few years. We would moreover like to thank
M.~Cafasso, I.~Krichever, P.~Lorenzoni, F.~Magri and all the participants of the
conference on ``Integrable Systems in Pure and Applied Mathematics'' at
Alghero, June 2010, for discussions and their interest in this project. We
would also like to thank the Complex Geometry Group of Universit\'e
Paris VI, the Departments of Mathematics of the Universities of Geneva, Glasgow and Milano-Bicocca, and in particular
G.~Falqui, X.~Ma, I.~Strachan and A.~Szenes for kind hospitality while this work was being prepared.
A.~B. was supported by a Postdoctoral Fellowship of the Swiss National Science
Foundation (FNS); partial support from the INdAM-GNFM Grant ``Teoria di
stringa topologica e sistemi integrabili'' is also acknowledged. G.~C. wishes to acknowledge the support of the Center for Mathematics of the University of Coimbra (CMUC), the Mathematical Physics sector of SISSA in Trieste and the Mathematics Department of University of Milano-Bicocca. 
P.~R. was supported by a Postdoc of Fondation de Sciences Math\'ematiques de Paris at the Institut de Math\'ematiques de Jussieu, UPMC, Paris 6.

\section{Ablowitz-Ladik and $2D$-Toda hierarchies}
\label{sec:hom}
%

\subsection{Ablowitz-Ladik system and Toeplitz lattice}

The complexified Ablowitz-Ladik (AL) system \cite{MR0377223} is given by the pair of equations
\bea
\dot{x}_n &=& \frac{1}{2}(1-x_n y_n)(x_{n-1}+x_{n+1})
, \nn \\
\dot{y}_n &=& -\frac{1}{2}(1-x_n y_n)(y_{n-1}+y_{n+1})
\label{eq:AL}
\eea
defining the time evolution of two sequences of complex variables $x_n$,  $y_n$ with $n\in\bbZ$.
The AL system admits an infinite number of conservation laws and is part of a hierarchy of mutually commuting evolutionary flows, usually described by semidiscrete zero-curvature equations~\cite{MR0377223, MR1993935}. 

In the semi-infinite case the AL hierarchy is equivalent, as noted in \cites{MR1794352, MR1993935} and shown in detail by Cafasso \cite{MR2534519}, to a
peculiar reduction of the 2D-Toda lattice hierarchy, called the Toeplitz lattice, which naturally arises in the study of the integrable dynamics of moment matrices associated with biorthogonal polynomials on the unit
circle and whose orbits are selected by Toeplitz initial data for an associated factorization problem~\cite{MR0810626}. In particular it describes the solution associated with a unitary matrix model~\cite{MR1794352}. 

Rather than dealing with the semi-infinite case, we present here a slightly more general definition of the Toeplitz lattice in the case of bi-infinite matrices, i.e. we assume that the matrix indices below span integer values, $n,m \in \bbZ$. 
This choice turns out to be somehow more natural, allowing us to easily identify the Toeplitz lattice with a rational reduction of the 2D-Toda hierarchy and to obtain the dispersionless limit that we need later. The semi-infinite case will be recovered as a further simple reduction (see  Appendix A for further details of this analysis in the case of the semi-infinite Toeplitz lattice). 

Recall that the 2D-Toda Lax matrices~\cite{MR810623} are given by
\beq
L_1 =  \Lambda + \sum_{j\leq 0}u^{(1)}_{j} \Lambda^{j}, \qquad
L_2 = u^{(2)}_{-1}\Lambda^{-1} + \sum_{j\geq 0}u^{(2)}_{j} \Lambda^{j} 
\label{eq:2dtlax}
\eeq
where $\Lambda$ is the shift matrix
\beq
\Lambda_{n,m}=\delta_{n+1,m} ,
\eeq
the diagonal matrices $u_j^{(i)}$ represent the dependent variables and the matrix indices $n,m \in \bbZ$.  The 2D-Toda flows can be written in the Lax form as
\beq \label{2dtlax}
\de_{s^{(1)}_j}L_i=\l[\l(L_1^j\r)_+,L_i\r], \quad \de_{s^{(2)}_j}L_i=\l[\l(L_2^j\r)_-,L_i\r], \qquad i=1,2
\eeq
where we denoted by $M_+$ (resp. $M_-$) the upper (resp. lower) diagonal part
of a matrix $M$, including (resp. excluding) the main diagonal.
\begin{defn}
We say that $L_1$ and $L_2$ are  {\rm Toeplitz} Lax matrices if they can be
written in the form\footnote{The horizontal and vertical lines separate the entries with negative and non-negative values of indices.}
\bes \label{L12-m}
\beq
L_1=
\left(
\begin{array}{ccc|ccccc}
\ddots 	& & & & & &  \\
&-x_{-1} y_{-2} 	&1	&0&0&0& \\
&- v_{-1} x_0 y_{-2} & - x_0 y_{-1} &1 &0&0& \\ \hline
&-v_{-1} v_0 x_1 y_{-2} &-v_0 x_1 y_{-1} &-x_1 y_0 & 1 &0 & \\
&- v_{-1} v_0 v_1 x_2 y_{-2} &- v_0 v_1 x_2 y_{-1} & -v_1 x_2 y_0 &-x_2 y_1 &1 & \\
&- v_{-1} v_0 v_1 v_2 x_3 y_{-2} &-v_0 v_1 v_2 x_3 y_{-1} & -v_1 v_2  x_3 y_0 & - v_2 x_3 y_1 &-x_3 y_2  & \\
&&&&&&\ddots
\end{array}
\right), 
\eeq
\beq
L_2=
\left(
\begin{array}{ccc|ccccc}
\ddots 	& & & & & &  \\
&-x_{-2}y_{-1} &-x_{-2}y_0 &-x_{-2}y_1 &-x_{-2}y_2 &-x_{-2}y_3 & \\
&v_{-1} &-x_{-1}y_0 &-x_{-1}y_1 &-x_{-1}y_2 &-x_{-1}y_3 & \\ \hline
&0&v_0 &-x_0y_1&-x_0y_2&-x_0y_3& \\
&0&0&v_1 &-x_1y_2 &-x_1y_3 & \\
&0 &0 &0 &v_2 &-x_2y_3 & \\
&&&&&&\ddots
\end{array}
\right) 
\eeq
\ees
where $x_n$, $y_n \in \bbC$, $n\in \bbZ$ and $v_n:=1-x_n y_n$.
\end{defn}
We will prove shortly that this is indeed a symmetry reduction of the 2D-Toda
lattice. \\

It is convenient to write these matrices in the equivalent form
\bes \label{L12-1}
\begin{align}
L_1 &= \Lambda - x^+  \left(1-(1-xy) \Lambda^{-1} \right)^{-1} y,  \\
L_2 &= (1-xy) \L^{-1} - x \left(1-\Lambda\right)^{-1} y^+ \label{L12-1b}
\end{align}
\ees
where $x$, resp. $y$, are diagonal matrices with entries given by $x_n$, resp. $y_n$, and $x^+$ denotes the shifted variable $x$, $\L x = x^+ \L$. 

Here and in the following the formal inverse of a matrix of the form $1-X$ is given by geometric series in $X$ and for this reason we sometimes denote it with $\frac1{1-X}$. Note that this is a proper (left and right) inverse of the bi-diagonal matrix $1-X$ with respect to the usual matrix multiplication\footnote{One should be aware of several fragile features of matrix multiplication when dealing with bi-infinite or semi-infinite matrices. In particular properties like associativity of the matrix product, existence and uniqueness of left and right inverses and their relation with the inverses of the corresponding linear map may not be taken for granted. See e.g.~\cite{MR0072965} for some examples.}. 

Note that we can also write
\beq \label{L12-2}
L_1 = \L \frac1v \l( 1-x \frac1{1-v\L^{-1}}y \r), \quad 
L_2 = \l(1- x \frac1{1-\L} y\r) \L^{-1}. 
\eeq

One can easily recognize $L_2$ to be a simple extension to the bi-infinite case of the semi-infinite version given in~\cite{MR1794352}. On the other hand $L_1$ is usually given in a dressed form. To see this, let $\ell$ be a diagonal matrix 
\beq
\ell=\diag\l(\dots,\ell_{-1},\ell_0, \ell_1, \dots\r)
\eeq
with entries that satisfy
\beq \label{h-rel}
\frac{\ell_{n+1}}{\ell_n}= 1- x_{n+1} y_{n+1} , \qquad n\in\bbZ . 
\eeq
Then
\begin{lem}
We have 
\[
\ell^{-1} L_1 \ell = \L \l( 1 - x \frac1{1-\L^{-1}} y \r)
\]
where $(1-\L^{-1})^{-1}$ is given by the matrix $\sum_{k\geq0} \L^{-k}$.
\end{lem}
\begin{proof}
A simple computation, rewriting~\eqref{h-rel} as
\beq
\ell \Lambda \ell^{-1} = \L \frac1{1-xy}, \qquad
\ell \L^{-1} \ell^{-1} = (1-xy) \L^{-1} . \nn 
\eeq \qedhere
\end{proof}
Explicitly
\beq
\ell^{-1} L_1 \ell=
\left(
\begin{array}{ccc|ccccc}
\ddots 	& & & & & &  \\
&-x_{-1} y_{-2} 	&v_{-1}	&0&0&0& \\
&-  x_0 y_{-2} & - x_0 y_{-1} &v_0 &0&0& \\ \hline
&- x_1 y_{-2} &- x_1 y_{-1} &-x_1 y_0 & v_1 &0 & \\
&-  x_2 y_{-2} &- x_2 y_{-1} & - x_2 y_0 &-x_2 y_1 &v_2 & \\
&- x_3 y_{-2} &- x_3 y_{-1} & -  x_3 y_0 & -  x_3 y_1 &-x_3 y_2  & \\
&&&&&&\ddots
\end{array}
\right) \nn
\eeq
which is the obvious extension of $\ell^{-1} L_1 \ell$ appearing in the semi-infinite Toeplitz lattice. 

We now show that the form of these matrices is preserved by 2D-Toda flows and that they correspond to the simplest rational reduction. This follows from  two simple observations.
\begin{prop} \label{prop-fact}
The Lax operators $L_i$ can be factorized as
\beq \label{ab}
L_1 = A B^{-1} , \qquad 
L_2 = B A^{-1}
\eeq
where the bi-diagonal matrices $A$ and $B$ are given by
\beq \label{aby}
A = -\frac1{y^+} (1- \L) y, \qquad
B = \frac1y (1-v\L^{-1}) y .
\eeq
\end{prop}
\begin{proof}
A simple computation, in the case of $L_2$ 
\begin{align}
B A^{-1} &= -\frac1y \l( 1- v\L^{-1} \r) \l( 1-\L \r)^{-1} y^+ \nn \\
&= \frac1y \l( v \L^{-1} (1-\L) - xy \r) (1-\L)^{-1} y^+ \nn \\
&= \frac{v}y \L^{-1} y^+ - x (1-\L)^{-1} y^+ \nn 
\end{align}
that is equal to~\eqref{L12-1b}. Notice that in the second equality we have used  the identity $\L^{-1} \L=1$, which is not satisfied in the semi-infinite case. 
\end{proof}
Explicitly
\beq
A =
\left(
\begin{array}{cc|ccc}
\ddots 	& & & &   \\
&- \frac{y_{-1}}{y_0} &1 &0 & \\ \hline
&0 &- \frac{y_0}{y_1} &1 & \\
&0 &0 & - \frac{y_1}{y_2} & \\
&&&&\ddots
\end{array}
\right), \quad
B =
\left(
\begin{array}{ccc|cc}
\ddots 	& & & &   \\
&- v_{-1}\frac{y_{-2}}{y_{-1}} &1 &0 & \\ \hline
&0 &-v_0 \frac{y_{-1}}{y_0} &1 & \\
&0 &0 & -v_1 \frac{y_0}{y_1} & \\
&&&&\ddots
\end{array}
\right). \nn
\eeq

We prove that rational Lax matrices~\eqref{ab} are invariant under the 2D-Toda flows following an argument similar to that of~\cite{MR1340299} for the rational reductions of the KP hierarchy. 
\begin{prop}
Given infinite matrices of the form
\beq
A = \Lambda + a, \quad
B = 1 + b \Lambda^{-1}
\eeq
for diagonal matrices $a$, $b$, the equations 
\bes \label{abflows}
\begin{align}
A_{s^{(1)}_i} = ((A B^{-1})^i)_+ A - A ( (B^{-1} A)^i )_+,
\quad
B_{s^{(1)}_i} = ((A B^{-1})^i)_+ B - B ( (B^{-1} A)^i )_+, \label{at}\\
A_{s^{(2)}_i} = ((B A^{-1})^i)_- A - A ( (A^{-1} B)^i )_-,
\quad
B_{s^{(2)}_i} = ((B A^{-1})^i)_- B - B ( (A^{-1} B)^i )_- 
\end{align}
\ees
are well-defined and induce the 2D-Toda equations~\eqref{2dtlax} on the Lax operators~\eqref{ab}.
\end{prop}
\begin{proof}
Let us check the first equation in ~\eqref{at}. Clearly the right-hand side is upper triangular. Rewriting it as 
\[
A ((B^{-1} A)^i)_- - ((A B^{-1})^i)_- A
\]
one concludes that it is actually diagonal, hence the equation is well-defined.
\end{proof}
\subsection{Hamiltonian formalism and dispersionless limit}
The bi-infinite Toeplitz flows can be cast in Hamiltonian form 
\bea
 \frac{\de x_n}{\de
s^{(k)}_i}=(1-x_ny_n)\frac{\de H^{(k)}_i}{\de y_n} ,  &~~~~~&
\frac{\de y_n}{\de s^{(k)}_i}=-(1-x_ny_n)\frac{\de
H^{(k)}_i}{\de x_n} 
\label{eq:lateom}
\eea
where the Hamiltonians
\beq
 H^{(k)}_i:=-
\frac{1}{i}\tr~{ L} _k^i,~~i=1,2,3,...,~~k=1,2
\eeq
mutually commute with respect to the symplectic structure
\beq \label{symplec}
\omega := \sum_{k\in\bbZ}
\frac{dx_k \wedge dy_k}{1-x_ky_k} .
\eeq

The first equations of the hierarchy, i.e. the Ablowitz-Ladik system~\eqref{eq:AL}, correspond to the combination of Hamiltonians
\beq
 H_{AL}:=\frac12 (H_1^{(1)}+H_1^{(2)})=-\frac{1}{2}\tr~ \l(L_1+L_2\r)=\frac{1}{2}\sum_{i\in\bbZ}
\l( x_{i+1}y_i+x_{i}y_{i+1}\r) .\nn
\eeq

To conclude our treatment of the Toeplitz lattice as rational 2D-Toda reduction we prove that 
\begin{prop}
The Hamilton equations~\eqref{eq:lateom} induce on the matrices~\eqref{aby} the flows~\eqref{abflows}.
\end{prop}
\begin{proof}
As noted in previous proof, both sides of the first equation in~\eqref{at} are diagonal, hence it can be written
\beq \label{hpf1}
\frac{\de}{\de s_i^{(1)}}\l(\frac{y_n}{y_{n+1}} \r) =
\frac{y_n}{y_{n+1}} \l( (A B^{-1} )^i - (B^{-1} A)^i \r)_{n,n} . 
\eeq
On the other hand the Hamilton equations~\eqref{eq:lateom} give
\beq \label{hpf2}
\frac{\de}{\de s_i^{(1)}}\l(\frac{y_n}{y_{n+1}} \r) = 
\frac{y_n}{y_{n+1}} \l( -\frac{v_n}{y_n} \frac{\de H_i^{(1)}}{\de x_n}
+ \frac{v_{n+1}}{y_{n+1}} \frac{\de H_i^{(1)}}{\de x_{n+1}} \r) .
\eeq
We need to check that~\eqref{hpf2} implies~\eqref{hpf1}. Using the fact that $L_1 = A B^{-1}$ and that $A$ does not depend on $x_n$ we can  compute
\[
\frac{\de H_i^{(1)}}{\de x_n} =  (y^- \L^{-1} B^{-1} L_1^i )_{n,n}
\]
and
\beq \label{hpf3}
 -\frac{v_n}{y_n} \frac{\de H_i^{(1)}}{\de x_n} 
 =-\l( \frac{v y^-}{y\ } \L^{-1} B^{-1} L_1^i \r)_{n,n}
 = \l( (AB^{-1})^i\r)_{n,n} -\l(B^{-1}(AB^{-1})^i\r)_{n,n},
\eeq
where, in the last equality, we have used the identity
\[
-\frac{v y^-}{y\ } \L^{-1}=B-1
\]
which follows from~\eqref{aby}.
Similarly
\begin{align}
\frac{v_{n+1}}{y_{n+1}} \frac{\de H_i^{(1)}}{\de x_{n+1}} &=
\l( \frac{v y^-}{y\ } \L^{-1} B^{-1} L_1^i \r)_{n+1,n+1}\nn \\
&=\l( B^{-1} L_1^i \frac{v y^-}{y\ } \L^{-1}\r)_{n,n} \label{hpf4}\\
&= - \l( (B^{-1}A)^i\r)_{n,n} +\l(B^{-1}(AB^{-1})^i\r)_{n,n} .\nn
\end{align}
Substituting~\eqref{hpf3} and~\eqref{hpf4} in~\eqref{hpf2} we conclude. The rest of the equations~\eqref{abflows} are obtained from the Hamilton equations with analogous computations which we leave as an exercise to the reader. 
\end{proof}

\begin{rem} \label{assoc}
We already mentioned that the multiplication might not be associative in the case of infinite matrices. 
Indeed, from the factorization~\eqref{ab} one is tempted to conclude that $L_1 L_2 = (A B^{-1}) (B A^{-1}) $ equals, assuming associativity, to $A (B^{-1} B) A^{-1} =1$. However it is easy to check from~\eqref{L12-m} that $L_1 L_2 \not= 1 \not= L_2 L_1$. A version of this constraint will nonetheless work in semi-infinite case.
\end{rem}

In view of the last remark, it is worth to point out that the semi-infinite
Toeplitz lattice departs slightly from the bi-infinite case, and it is
interesting to investigate on its own especially in view of its connection to
unitary matrix models \cite{MR1794352}. We give the
details of this case in Appendix \ref{sec:semiinf}. \\

Following \cite{MR2462355}, we introduce the pair of variables
\begin{equation}\label{eq:vw1}
\begin{split}
w &= \log \l(1-x y\r), \\
v &= \frac{1}{2}\l(1-\Lambda^{-1}\r)(\log x-\log y).
\end{split}
\end{equation}
With this choice of dependent variables, the Lax matrices can be rewritten as 
\begin{equation}\label{eq:laxvw}
\begin{split}
L_1 &=  \Lambda \re^{-w} \l(1- \sqrt{1-\re^w} \frac{1}{1-\re^{v+w}\Lambda^{-1}} 
\sqrt{1-\re^w}\r) ,\\
L_2 &= \l(1- \sqrt{1-\re^w} \frac{1}{1-\Lambda \re^{-v}}
\sqrt{1-\re^w}\r) \Lambda^{-1} .
\end{split}
\end{equation}

In this case
\[
L_1 = A B^{-1}, \quad L_2 = B A^{-1} 
\]
for
\[
A = - \frac{\re^{v^+}}{\sqrt{1-\re^{w^+}}}
\left( 1- \re^{-v^+} \L \right), \quad
B = \frac1{\sqrt{1-\re^w}}
\left( 1- \re^{v+w} \L^{-1} \right) .
\]

These matrices can be alternatively seen as formal difference operators acting on the real line and, correspondingly, the dependent variables $v$ and $w$ as functions of a space variable $x$ (not to be confused with the dependent variable denoted above with the same symbol). 

This observation allows us to straightforwardly obtain the long-wave
limit of the Toeplitz lattice in Lax form. The symbols of the Lax operators~\cite{MR1346289} 
\[
\lambda_i(p) =\sigma_{L_i}(p),\quad i=1,2
\]
are given by rational functions
\beq \label{laxsymb}
\lambda_1(p)=p\l(\frac{p-\re^v}{p-\re^{v+w}}\r)=(\lambda_2(p))^{-1} ,
\eeq
hence the dispersionless Lax equations can be compactly written in terms of the Lax symbol $\lambda(p) := \lambda_1(p)$ as
\beq \label{lax-form}
\frac{\de \lambda}{\de s^{(1)}_n}=\{ \l(\lambda^n\r)_+, \lambda \}_{[L]}, \quad \frac{\de
  \lambda}{\de s^{(2)}_n}=\{ \l(\lambda^{-n}\r)_- , \lambda \}_{[L]}, \quad i=1,2 \quad, 
\eeq
where $()_\pm$ denote the projections to the analytic and principal part and the Poisson bracket $\{, \}_{[L]}$ is defined as
\beq
\{a(p, x), b(p,x)\}_{[L]}:=p \frac{\de a(p,x)}{\de p}\frac{\de b(p,x)}{\de x}-p\frac{\de a(p,x)}{\de x}\frac{\de b(p,x)}{\de p} .
\eeq

\subsection{A conformal Frobenius manifold for the AL hierarchy}

The dispersionless Lax formalism for the Toeplitz reduction of 2D-Toda paves
the way to canonically associate a Frobenius manifold with the AL
hierarchy. We denote by $M_{g;n_1,\dots,n_m}$ the Hurwitz space
\bea
M_{g;n_1, \dots,n_m} &=\big\{& (\Gamma;p_1, \dots, p_m;f): \Gamma \hbox{
  smooth projective, } \dim_{\bbC}\Gamma=1, h^{1,0}(\Gamma)=g, \nn \\
& & f:\Gamma\to \bbC\bbP^1 \in \cO_{\Gamma \setminus \{p_1, \dots, p_m \}}, e_f(p_j)=n_j\big\}/ \mathord{\sim}
\eea

where the quotient is under biholomorphic equivalence. We can view the
dispersionless Lax operator 
\beq
\lambda(p)=p+\re^v(\re^w-1)+\re^{2v+w}\frac{\re^w-1}{p-\re^{v+w}} 
\label{eq:spothom}
\eeq
as the datum of a degree 2 covering map $\lambda: \bbC\bbP^1 \to \bbC\bbP^1$
which is unramified at infinity, that is, $[(\lambda(p),
\bbC\bbP^1)]/\mathord{\sim} \in M_{0;1,1}$,  
where we pick an equivalence class
under M\"obius transformation in the form \eqref{eq:spothom}. We have in this
case
\beq
p_1=\infty, \quad p_2=\re^{v+w}.
\eeq  
By regarding \eqref{eq:spothom} as the tree-level superpotential of a topological Landau-Ginzburg model
\cites{Dubrovin:1992eu, Dubrovin:1994hc, Krichever:1992qe}, we can associate a
Frobenius structure with $M_{0;1,1}$ as follows. Let $[\Gamma,f] \in M_{g,n_1,
  \dots,n_m}$, $0<D<\sum_{i=1}^m n_i p_i $ a divisor on $\Gamma$ and $\rd \omega
\in H^{1,0}_{\bar\de}(\Gamma \setminus D)$ a meromorphic differential,
possibly with poles at $p_i$ of orders less than $n_i+1$. The pair $(M_{g;n_1,
\dots n_m},\rd\omega)$ can be endowed with the structure of a Frobenius manifold through
the Landau-Ginzburg formulae \cites{Cecotti:1992rm, Dubrovin:1994hc}
\bea
\label{eq:etalg}
\eta(\de_i, \de_j) &=& \sum \Res_{\rd \lambda=0}\l\{ \frac{\de_i \lambda(p)
  \de_j \lambda(p)}{\lambda'(p)} \frac{\rd p}{p^2}  \r\} \\
c(\de_i, \de_j, \de_k) &=& \sum \Res_{\rd \lambda=0}\l\{ \frac{\de_i \lambda(p)
  \de_j \lambda(p) \de_k \lambda(p)}{\lambda'(p) } \frac{\rd p}{p^2} \r\}
\label{eq:clg}
\eea
In order for $(M_{g;n_1,
\dots n_m},\rd\omega)$  to satisfy all axioms of a Frobenius manifold, $\rd \omega$
should fall in one of five different categories of meromorphic 1-forms, which were
characterized in detail in \cites{Dubrovin:1992eu, Dubrovin:1994hc}; such
1-differentials go under the name of {\it admissible primary
  differentials}. We refer the reader to \cite{Dubrovin:1994hc} for more
details, and concentrate on the case of $M_{0;1,1}$ in the following. \\ 
For the case of $M_{0;1,1}$, Dubrovin's classification reduces to one case:
$\rd\omega$ is the unique meromorphic third kind differential with
\beq
\Res_{p=p_1}\rd \omega=1, \quad \Res_{p=p_2} \rd \omega=-1,
\eeq 
i.e., when $p_1=\infty$,
\beq
\rd \omega= \l\{\bary{ccc}\frac{p \rd p}{p_2(p-p_2)} & \mathrm{for} & p_2 \neq
0 \\ -\frac{\rd p}{p} & \mathrm{for} & p_2 =
0\eary\r. 
\eeq
In this case, the Frobenius manifold induced by \eqref{eq:etalg}, \eqref{eq:clg}  is the one associated with the Extended
Toda hierarchy~\cite{MR2108440}, which is in turn related to the
Gromov-Witten theory of the projective line. In our case, the Toeplitz reduction of
2D-Toda \eqref{eq:hnres} binds us to take as primitive form
\beq
\rd \omega = \frac{\rd p}{p}
\eeq
which is {\it not} admissible; as a consequence, moving from 1D-Toda to AL
implies that the solution of WDVV associated
with \eqref{eq:etalg}, \eqref{eq:clg} will {\it not} satisfy all axioms of a
Frobenius manifold. With a slight abuse of language, we will sometimes
refer
to this weaker structure\footnote{A convenient name could be ``almost-Frobenius
manifold'', as the type of solution of WDVV bears many resemblances with
those considered in \cite{MR2070050}, albeit differing in one important
aspect (namely $E\neq e$). Still, as Dubrovin's ``almost-duality''
will play a different role elsewhere in the text, we will refrain from doing so.} induced on $M_{0;1,1}$ still as a ``Frobenius manifold''.
\\
 
This section is devoted to a thorough characterization of this canonical
Frobenius structure associated with the AL hierarchy. We have the following
\begin{thm}
Eq. \eqref{eq:etalg}, \eqref{eq:clg} endow the Hurwitz space $M_{0;1,1}$ with
the structure of a charge $d=1$, non-degenerate, semi-simple Frobenius manifold $\cM_{AL} := (M_{0;1,1}, e,
E, \eta, F_0)$ with a {\rm non-covariantly
constant} unit $e$. In flat co-ordinates $t_1$, $t_2$ for the metric $\eta$, the
prepotential reads
\beq
F_0=\frac{1}{2} t_2 t_1^2+\re^{t_2} t_1+\frac{1}{2} t_1^2 \log
   t_1
\label{eq:f0hom}
\eeq
whereas the unit $e$ and the Euler vector field $E$ are given as
\bea
\label{eq:ehom}
e &=& \frac{t_1 \de_{t_1}-t_2}{t_1-\re^{t_2}}, \\
E &=& t_1 \de_{t_1}+t_2. 
\label{eq:Ehom}
\eea
\end{thm}
\pf The proof follows from a straightforward calculation from
\eqref{eq:etalg}, \eqref{eq:clg}. We reproduce here the main steps. 

It is immediate to check that the metric $\eta$ is flat. Introducing
co-ordinates $(t_1,t_2)$ such that
\beq
\re^v=\re^{t_2}-t_1, \quad w=t_2-\log(\re^{t_2}-t_1),
\label{eq:vwt1t2}
\eeq
%
the metric $\eta$ takes the off-diagonal form
\beq
\eta\l(\de_{t_i}, \de_{t_j}\r)=\delta_{i+j,3}.
\eeq
In this co-ordinates, the Landau-Ginzburg formula \eqref{eq:clg} for the structure constants yields the
expression \eqref{eq:f0hom} for the prepotential. \\ 

As for the usual theory of Frobenius structures on Hurwitz spaces, the
critical values of the superpotential give a set of canonical co-ordinates of the
Frobenius manifold. Denoting by $q_{1,2}$ the critical points of $\lambda(p)$,
\bea
q_1 &=& \re^{\frac{t_2}{2}}\l(\re^{\frac{t_2}{2}}+\sqrt{t_1}\r), \nn \\
q_2 &=& \re^{\frac{t_2}{2}}\l(\re^{\frac{t_2}{2}}-\sqrt{t_1}\r),
\eea
canonical co-ordinates are given as
\bea
u_1 &=& \lambda(q_1)=\l(\re^{\frac{t_2}{2}}+\sqrt{t_1}\r)^2,\nn \\
u_2 &=& \lambda(q_2)=\l(\re^{\frac{t_2}{2}}-\sqrt{t_1}\r)^2,
\label{eq:cancoords}
\eea
and it is straightforward to check that the corresponding vector fields give
idempotents of the algebra \eqref{eq:clg}
\beq
\de_\g u c^\g_{\a\b} = \de_\a u \de_\g u.
\eeq
In particular, the Frobenius algebra induced on the tangent bundle of $M_{0;1,1}$ is generically semi-simple. \\

With this ingredients at hand, we can readily determine the expression for the
unit $e$ and the Euler vector field $E$. By definition, we have
\beq
e=\de_{u_1}+ \de_{u_2}
\eeq
and \eqref{eq:cancoords} implies \eqref{eq:ehom}. On the other hand, we know
that in the usual theory of Frobenius manifolds associated with Hurwitz spaces, the vector
\beq
E:=\sum_{i=1,2}u_i \de_{u_i}
\label{eq:Egen}
\eeq
is the Euler vector field of the Frobenius manifold. For the case at hand,
\eqref{eq:Egen} becomes, in flat co-ordinates
\beq
E=t_1 \de_{t_1}+ \de_{t_2}.
\eeq
This is indeed the Euler vector field for the solution of WDVV
\eqref{eq:f0hom}. Up to quadratic terms, we have explicitly
\beq
L_E F_0 = 2 F_0 = (3-1) F_0,
\eeq
namely, the Frobenius structure is quasi-homogeneous, with charge
$d=1$. Its non-degeneracy 
\beq
[e,E]=e
\eeq
follows trivially from \eqref{eq:ehom}, \eqref{eq:Ehom}.
\qed
\\
\begin{rmk}
As compared to the classical definition of a Frobenius manifold, we see that
the axiom of covariant constancy of the unit vector field with respect to
the Levi-Civita connection of $\eta$
\beq
\nabla e = 0
\eeq
is violated by \eqref{eq:f0hom}. In particular
\beq
\de_{\a}\de_{\b} L_e F_0 \neq \eta_{\a\b}
\eeq
Somewhat remarkably, though, the grading axiom, which states that the Euler vector field is linear
\beq
\nabla\nabla E = 0
\eeq
is instead respected, as is manifest from \eqref{eq:Ehom}. 
\end{rmk}

\subsection{Bi-Hamiltonian structure}

Denote by $\{,\}_i$, $i=1,2$ the Poisson brackets of hydrodynamic type on the loop space $\mathcal{L} (\mathcal{M}_{AL})$ associated with the metric $\eta$ and the intersection form $g$ respectively. Recall that the intersection form is the bilinear pairing on the cotangent bundle $T^* \mathcal{M}_{AL}$ defined by
\beq \label{interformg}
g(w_1, w_2) := i_E (w_1 \cdot w_2 )
\eeq
where the product of the $1$-forms $w_i$ is induced on $T^* \mathcal{M}_{AL}$ by the Frobenius algebra on the tangent by the map $\eta^{}: T \mathcal{M}_{AL} \rightarrow T^* \mathcal{M}_{AL}$.

It is a general result of the theory of Frobenius manifolds that the contravariant metrics $\eta$ and $g$ form a flat pencil; this in particular implies that the associated Poisson brackets $\{,\}_i$ are compatible. In the present case the compatibility is confirmed by a straightforward computation.

In flat coordinates the Poisson brackets are given by
\[
\{t_1(x),t_2(y)\}_1 = \delta'(x-y), \quad
\]
with the other entries equal to zero, and
\begin{subequations}
\begin{align}
&\{ t_1 (x) , t_1 (y) \}_2 = 2 t_1 \re^{t_2} \delta'(x-y) 
+ (t_1 \re^{t_2})' \delta(x-y), \\
&\{ t_1(x) , t_2(y) \}_2 = (t_1+\re^{t_2})\delta'(x-y) 
+ (t_1 + \re^{t_2})' \delta(x-y), \\
&\{ t_2(x) , t_2(y) \}_2 = 2 \delta'(x-y) .
\end{align}
\end{subequations}

\begin{rmk}
The Poisson pencil $\{,\}_1 + z \{,\}_2$ is not exact: it can be easily proved that there is no vector field $X$ such that 
\begin{equation} \label{exactness}
\mathrm{Lie}_X \{,\}_2 = \{,\}_1, \quad
\mathrm{Lie}_X \{,\}_1 =0 . 
\end{equation}
This fact is a direct consequence of dropping the axiom of flatness of $e$:
indeed, all Poisson pencils associated with Frobenius manifolds with flat unit are exact; in such a case a vector field $X$ such that~\eqref{exactness} holds is given by the unit $e$.
\end{rmk}

\begin{rmk}
Note that we do not claim any relation of these Poisson structures with the Poisson structures of 2D-Toda~\cite{GC} or with the symplectic form~\eqref{symplec} of the Toeplitz lattice. It would be interesting to obtain the Poisson pencil presented here as a reduction of the 2D-Toda Poisson pencil, or to obtain dispersive counterparts of the Poisson brackets $\{,\}_i$. 
\end{rmk}

Considering the last remark is somehow unexpected that the 2D-Toda Hamiltonians and the Poisson brackets given above provide the correct flows. Denote
\beq
\HH^{(i)}_n = \int h^{(i)}_n(v,w) d x, 
\eeq
where $h^{(i)}_n$ for $n\geq1$, $i=1,2$ are the dispersionless Hamiltonian densities obtained by restriction of the dispersionless 2D-Toda Hamiltonian densities to the submanifold of symbols of the form~\eqref{laxsymb}, i.e.
\beq \label{eq:hnres}
h^{(1)}_n = - \Res_{p=\infty} \frac{\lambda^n}{n} \frac{dp}p, \qquad
h^{(2)}_n = \Res_{p=0} \frac{\lambda^{-n}}{n} \frac{dp}p .
\eeq

These densities can be written in closed form in terms of hypergeometric functions. We have
\bea
h^{(1)}_n &=& -\Res_{p=\infty}\frac{\lambda^n}{n} \frac{\rd p}{p}=
-\Res_{p=\infty}\frac{1}{n}\l(p\frac{p-\re^v}{p-\re^{v+w}}\r)^n \frac{\rd p}{
  p} =  \frac{\re^{nv}}{n \ n!}\frac{\rd^n}{\rd x^n}\l(\frac{1-x}{1-x\re^w}\r)^n\Bigg|_{x=0}
\nn \\
&=& \frac{\re^{nv}}{n \ n!} \sum_{k=0}^n\l(\bary{c}n\\k\eary \r)\frac{\rd^{n-k}}{\rd x^{n-k}}(1-x)^n
\frac{\rd^k}{\rd x^k}(1-x \re^w)^{-n}\Bigg|_{x=0} \nn \\
&=&  \frac{(-1)^{n} \re^{nv}}{n}\sum_{k=0}^n\l(\bary{c}n\\k\eary
\r)\l(\bary{c}n+k-1\\k\eary \r) (-\re^w)^k \nn \\
&=&  \frac{(-1)^{n} \re^{nv}}{n} \,
_2F_1\left(-n,n;1;\re^w\right) \nn 
\eea 
and, by a similar computation,
\[
h^{(2)}_n = \frac{(-1)^n \re^{-nv}}{n}  \ _2 F_1(n,-n,1;\re^w).
\]

As an example, the dispersionless Ablowitz-Ladik Hamiltonian reads
\beq
\HH_{AL}:=-\frac{1}{2}\int \l(h^{(1)}_1(v,w)+h^{(2)}_{1}(v,w)\r)\rd x =\int (1-\re^w) \cosh v
\ \rd x .
\eeq

We have:
\begin{prop}
The dispersionless AL flows~\eqref{lax-form} admit the following Hamiltonian formulation
\[
\frac{\partial}{\partial s^{(1)}_n} \cdot = \{ \cdot , \HH^{(1)}_{n+1} \}_1 \nn , \quad
\frac{\partial}{\partial s^{(2)}_n} \cdot = \{ \cdot, \HH^{(2)}_n \}_1 \nn ,
\]
for $n>0$.
\end{prop}
Note the somewhat surprising relation of the 2D-Toda Hamiltonians with the restrictions of the 2D-Toda Lax flows. Even more surprisingly:
\begin{prop}
The Hamiltonians $\HH^{(i)}_n$ satisfy the following recursion relations for $n>0$
\bea
\{ \cdot , \HH^{(1)}_n \}_2 = \{ \cdot , \HH^{(1)}_{n+1} \}_1, \label{bhr1}  \\
\{ \cdot , \HH^{(2)}_{n+1} \}_2 = \{ \cdot , \HH^{(2)}_n \}_1 . \label{bhr2} 
\eea
\end{prop}

The first sequence of Hamiltonians $\HH^{(1)}_n$ is obtained by the
bi-Hamiltonian recursion~\eqref{bhr1} starting from the Casimir $ \HH^{(1)}_1
= \int t_1 \ dx$ of $\{,\}_1$. The second recursion involves the Hamiltonians
$\HH^{(2)}_n$ in a somewhat opposite order; moreover this second chain does
not contain the  Casimir $\int t_2 \ dx$, which turns out to be a Casimir of
both Poisson brackets, a phenomenon related to the resonance of the spectrum
of $\mathcal{M}_{AL}$. \\

As in the case of Frobenius manifolds with flat unit, one can define the
deformed flat connection $\tilde\nabla$ of $\mathcal{M}_{AL}\times\bbC^*$ and construct a
Levelt basis of deformed flat coordinates $\theta_{\alpha}(\zeta)$ which provide the Hamiltonians of the associated Principal hierarchy on the loop space $\mathcal{L} ( \mathcal{M}_{AL})$ .

They are given by
\[
\bar h_{\alpha,p} = \int \theta_{\alpha,p+1} \ dx, \quad \alpha=1,2, p\geq-1
\]
where the densities are obtained by expanding 
\[
\theta_\alpha(\zeta) = \sum_{p\geq0} \theta_{\alpha,p} \zeta^p 
\]
in the deformation parameter $\zeta$. An explicit computation shows that the generating function of the 
densities $h_n^{(1)}$ obtained in the reduction from 2D-Toda
\bea
f(\zeta) &=& \sum_{p=0}^\infty \frac{h_{p+1}^{(1)}}{p!} \zeta^{p},
\eea
has horizontal differential w.r.t. the extended deformed connection
on $\cM_{AL} \times \bbC^*$
\beq
\tilde\nabla d f{(\zeta)} = 0
\label{eq:defflat} .
\eeq
%
%
%
At the level of the generating function we have 
\beq
f(\zeta) = (1-\re^w) \re^{v} \Psi_2\l(1; 1,2; 
\zeta \re^{v}(1-\re^w)  ,-\re^{w+
  v}\zeta\r)
\label{eq:fpm}
\eeq
where we denoted by $\Psi_2(a;b,c;x,y)$ the generalized hypergeometric Humbert function~\cite{MR0058756}
\beq
\Psi_2(a; b,c; x, y) := \sum_{l,m=0}^\infty \frac{(a)_{l+m}}{(b)_l (c)_m}
\frac{x^l y^m}{l! m!}.
\label{eq:phi2}
\eeq
and with $(a)_n$ the Pochhammer symbol
$\Gamma(a+n)/\Gamma(a)$. The leading order in the $\zeta$-expansion of
\eqref{eq:fpm} shows that $f(\zeta)$ yields the deformed flat co-ordinate
$\theta^{1}(\zeta)=\theta_{2}(\zeta)$, hence in this case 
\[\theta_{2,p} = \frac{h^{(1)}_{p+1}}{p!}.\]

For the other co-ordinate, by
solving recursively the deformed flatness equations, we obtain 
\begin{align}
&\theta_{1,0}=t_2, \nn \\
&\theta_{1,1}=\re^{t_2} +t_1(t_2+\log t_1 -1), \nn \\
&\theta_{1,2}=\frac{t_1}4 ( 2 (2\re^{t_2}+t_1) \log t_1 + t_1 (2 t_2 -1) + 4 \re^{t_2} (t_2 -1) ) + \frac14 \re^{2 t_2}. \nn
\end{align}
From the general theory it follows that the Hamiltonians $\bar h_{\alpha,p}$ satisfy the bi-Hamiltonian recursion relations
\begin{align}
&\{ \cdot , \bar h_{1,p-1} \}_2 = 2 \{ \cdot , \bar h_{2,p-1} \}_1 + p \{ \cdot, \bar h_{1,p} \}_1, \nn \\
&\{ \cdot, \bar h_{2,p-1} \}_2 = (p+1) \{ \cdot , \bar h_{2,p} \}_1 \nn
\end{align}
for $p\geq0$. The recursion relation for the first set of Hamiltonians takes into account the resonance of spectrum of $\mathcal{M}_{AL}$ mentioned above. 

Another remarkable fact related to the non-flatness of $e$ is that the momentum functional $p=t_1 t_2$ generating the $x$-translations does not appear among the Hamiltonians densities $\theta_{\alpha, p}$ of the Principal hierarchy.

\section{Mirror symmetry for local $\bbC\bbP^1$}
\label{sec:inhom}
\subsection[Dubrovin's almost duality and a logarithmic LG mirror]{Dubrovin's
  almost duality and a logarithmic Landau-Ginzburg mirror}

In the
light of our findings in Section \ref{sec:hom}, it is
natural to ask whether the Frobenius structure associated with the Gromov-Witten
theory of the resolved conifold has anything to do with the one in \eqref{eq:f0hom} and, if so,
whether we can learn anything new about the former from our discussion of the Toeplitz
reduction and its dispersionless limit. We now turn to answer both
questions in the affirmative. 

A key role in the discussion to follow will be played by Dubrovin's notion
\cite{MR2070050} of ``duality of (almost)-Frobenius manifolds'', which we
briefly recall here. 
Let $\cM := (M, e, E, \eta, F_0)$ be a Frobenius manifold, with
unit $e$, Euler vector field $E$, flat invariant pairing $\eta$ and structure
constants $c_{\a,\b}^\g=\eta^{\g\d}\de^{3}_{\a\b\d}F_0$. 
As in~\eqref{interformg} we associate with this data a bilinear form $g$ on $T^* M$, called intersection form. On the complement of the discriminant, i.e. the analytic subset $\mathrm{discr} M \subset M$ where $g$ is degenerate, 
the inverse of the intersection form defines a second flat metric (we still denote it by $g$). 
We can associate with $\cM$ another solution of WDVV, which does
not in principle satisfy all axioms of a Frobenius manifold.
\begin{defn}
The Dubrovin dual $\widehat{\cM}$ of a Frobenius manifold $\cM$ is the quadruplet
$(\widehat{M}, E, g,\widehat{F_0})$, where $\widehat{M}= M\setminus
\mathrm{discr} M$, $E$ is the Euler vector field on $M$, $g$ is the second
metric. In flat co-ordinates $p_i$ for $g$, $\widehat{F_0}$
is defined as to satisfy
\beq
\frac{\de^3 \widehat{F_0}}{\de p_i \de p_j \de
  p_k}=G_{ia}G_{jb}\frac{\de t_\g}{\de p_k} \frac{\de p_a}{\de t_\a}\frac{\de p_b}{\de t_\b}c_\g^{\a\b}
\label{eq:dualf0}
\eeq
where $G_{ij}$ is the Gram matrix of the metric $g$.
\end{defn}

\begin{thm}[Dubrovin, \cite{MR2070050}]
The dual prepotential \eqref{eq:dualf0} induces a commutative, associative
product $\star:T\widehat{M}\otimes T\widehat{M}\to T\widehat{M}$,
\beq
\de_i \star \de_j = E^{-1}\cdot \de_i \cdot \de_j
\label{eq:dualprod}
\eeq
under which the intersection pairing is invariant
\beq
g(\de_i \star \de_j, \de_k) = g(\de_i, \de_j  \star \de_k) .
\eeq
In particular, the Euler vector field on $M$ is the identity of the dual
product on $\widehat{\cM}$.
\end{thm}

\begin{rmk}
The solutions of WDVV obtained by the duality
\eqref{eq:dualprod} do not fulfill all axioms of a Frobenius manifold. First
of all, the
Euler vector field - that is, the dual unity field - need not be covariantly
constant under the Levi-Civita connection of $g$. Secondarily, when the charge $d$
of $\cM$ is different from 1, the dual prepotential is homogeneous of degree $(1-d)$
under $L_E$ \cite{MR2070050}, but it need not satisfy a quasi-homogeneity
condition if $d=1$. 
Note that in our case, while the dual prepotential will indeed fail to be homogeneous, the dual unit vector field will turn out to be nonetheless covariantly constant.
\end{rmk}

\begin{rmk}
It should be stressed that in the definition \eqref{eq:dualprod} of the dual
product, and in the proof of its associativity, no reference is made to the
fact the unit $e$ of $\cM$ be constant in flat co-ordinates $t_\a$. In other
words, the notion of Dubrovin-duality generalizes to the case in which $e$ is
not covariantly constant under the Levi-Civita connection of $\eta$.
\end{rmk}

When a Landau-Ginzburg description of $\cM$ is available we can obtain a
rather compact picture of Dubrovin's duality. It is straightforward to show
\cite{MR2070050} that the intersection pairing and dual product are obtained
by sending $\lambda\to\log\lambda$ in \eqref{eq:etalg}, \eqref{eq:clg}:

\bea
\label{eq:glg}
g(\de_i, \de_j) &=& \sum \Res_{\rd \lambda=0}\l\{ \frac{\de_i \log \lambda(p)
  \de_j \log \lambda(p)}{\lambda'(p)} \frac{\lambda(p)\rd p}{p^2} \r\} \\
\widehat{c}(\de_i, \de_j, \de_k) &=& \sum \Res_{\rd \lambda=0}\l\{ \frac{\de_i \log\lambda(p)
  \de_j \log \lambda(p) \de_k \log \lambda(p)}{\lambda'(p)} \frac{\lambda(p) \rd p}{p^2}  \r\}
\label{eq:cdlg}
\eea
where the sums run over critical points of the superpotential $\lambda$. \\

It is natural to conjecture that the notion of Dubrovin-duality could be the key to connect the Toeplitz lattice
hierarchy to the topological hierarchy of  \cite{Brini:2010ap}. Indeed,
consider the $T\simeq\bbC^*$-equivariant Gromov-Witten theory of a toric variety $X$,
where $T$ acts on $X$ with compact fixed loci. Then the genus zero primary $T$-equivariant Gromov-Witten potential of~$X$
\bea
F_0^X &=&\sum_{n=0}^\infty\sum_{\b\in H_2(X,\bbZ)}\sum_{\a_1, \dots, \a_n}
\frac{t_{\a_1}\dots t_{\a_n}}{n!}\bra\phi_{\a_1} \dots \phi_{\a_n}
\ket_{0,n,\beta}^X \\
\bra\phi_{\a_1} \dots \phi_{\a_n}
\ket_{g,n,\beta}^X &=& \int_{[\cM_{g,n,\beta}]^{\mathrm{vir}}} \prod_{i=1}^n
\mathrm{ev}^*_i \phi_{\a_i}, \qquad \phi_{j} \in H_T^\bullet(X, \mathbb{C}), \phi_{1}=\mathbf{1}
\eea
is, as in ordinary non-equivariant Gromov-Witten theory, a solution of WDVV for which the fundamental
class and point splitting axiom \cite{Kontsevich:1994qz}
pin down the direction of the unit $\mathbf{1} \in H^\bullet(X)$ as the one
that induces the Poincar\'e pairing on $H^\bullet(X)$
\beq
\frac{\de^3 F_0^X}{\de t_1 \de t_{\a}\de t_{\b}} = \int_X \phi_\a \cup \phi_\b
= \eta(\de_\a, \de_\b)
\eeq
i.e., a flat invariant pairing on $TH^\bullet(X)$. In other words, in
presence of a torus action the tangent bundle
of the equivariant quantum co-homology $\cM_X=QH_T^\bullet(X)$ is again endowed with
the structure of a commutative, associative algebra with a covariantly
constant unit. What departs from the ordinary theory of Frobenius manifolds is
the existence of an Euler vector field, as the degree axiom of Gromov-Witten theory breaks down, due
to the non-trivial grading of the ground ring $\bbC(\nu)$ of $QH^\bullet_T(X)$. As a consequence, the genus zero equivariant Gromov-Witten
potential of $X^{\circlearrowleft T}$ is still a solution of WDVV, but it fails to be quasi-homogeneous. \\

As was discussed in detail in \cite{2010arXiv1006.0649M}, the Dubrovin-duals
of charge $d=1$ Frobenius manifolds are solutions of WDVV with covariantly
constant unit, whereas the dual Euler vector field is ill-defined. They are
therefore the natural structures to look at in order to connect our results in
Section \ref{sec:hom} to the topology of moduli spaces. We have indeed the following 

\begin{thm}
The Dubrovin dual $\widehat{\cM}_{AL}$ of the Frobenius manifold $\cM_{AL}$
associated with the Toeplitz reduction of 2D-Toda is the Frobenius algebra
structure induced on $TM_{0;1,1}$ by the dual prepotential
\beq
\widehat{F_0}=\frac{1}{2}v^2 w + \Li_3(\re^w),
\label{eq:f0inhom}
\eeq
with constant unit $E=\de_v$.
In \eqref{eq:f0inhom}, $v$ and $w$ are  flat co-ordinates for the intersection
form and are defined as in \eqref{eq:vw1},
\eqref{eq:spothom}, \eqref{eq:vwt1t2}
\begin{align}
t_2 &= v+w, \nn \\
t_1 &= \re^v (\re^w -1). \nn
\end{align}
\end{thm}
\pf The proof follows from a straightforward calculation from
\eqref{eq:spothom}, \eqref{eq:glg} and \eqref{eq:cdlg}. \\ \qed \\
The prepotential \eqref{eq:f0inhom} coincides with the genus zero
Gromov-Witten potential of $\cO_{\bbP^1}(-1)_{[\nu]}$ $\oplus$ $\cO_{\bbP^1}(-1)_{[-\nu]}$
\cite{Brini:2010ap}, upon sending $v\to\ri \frac{v}{\nu}$, where $\nu$
is the equivariant parameter of the anti-diagonal $\bbC^*$-action on
$\cO_{\bbP^1}(-1)\oplus\cO_{\bbP^1}(-1)$. The combination of the Toeplitz
reduction with Dubrovin's duality therefore yields a mirror Landau-Ginzburg
description of the Gromov-Witten theory of local $\bbC\bbP^1$, with a logarithmic superpotential
\beq
\widehat{\lambda}(p)=\log\l(p\frac{p-\re^v}{p-\re^{v+w}}\r).
\label{eq:spotinhom}
\eeq
%

\subsection{Almost duality and twisted Picard--Lefschetz theory}
We want to compute a system of deformed flat coordinates for the Frobenius
manifold $\widehat{\cM}_{AL}$ using oscillating integrals for the
Landau-Ginzburg model given by the superpotential \eqref{eq:spotinhom}.

Such model corresponds to the twisted Picard--Lefschetz theory of the
meromorphic function $\lambda_{v,w}(p)=p\frac{p-\re^v}{p-\re^{v+w}}$. In
particular it corresponds to considering cycles in the complement of
$\lambda_{v,w}^{-1}(0)$, endowed with a local system of coefficients
transforming nontrivially upon circuit around such hypersurface, as opposed to
ordinary Picard--Lefschetz theory, which considers cycles on
$\lambda_{v,w}^{-1}(0)$. We will here review the basics of twisted
Picard--Lefschetz theory, referring the reader to \cite{MR936695} for more details.\\

Let us denote by $\pi:X\to\mathbb{P}^1\setminus\{0,\re^v,\re^{v+w},\infty\}$ the
cover where $\widehat{\lambda}(p)$ is defined (an infinite number of sheets
joint at the branch cuts $[0,\re^{v+w}]$, $[\re^v,\infty]$). The oscillating
integral formula for the deformed flat coordinates of $\widehat{\cM}_{AL}$ is
\beq \mathfrak{p}_\alpha(z):=\frac{1}{2\pi\ri}\int_{\gamma_\alpha} \re^{z\widehat{\lambda}(p)} \frac{\di
  p}{p}, \qquad \alpha=1,2.
\label{eq:loopint}
\eeq
The integration cycles $\gamma_1$, $\gamma_2$ are a basis of the homology with
local coefficients
$H_1(\mathbb{P}^1\setminus\{0,\re^v,\re^{v+w},\infty\},\{0,\re^v\};\mathbf{L}(q))$.

In general, for a superpotential $\widehat{\lambda}=\log(\lambda(x,a))$ with
$\lambda:\mathbb{C}^n\times\mathbb{C}^\mu\to\mathbb{P}^1$ and
$\lambda_a:=\lambda(\cdot,a)$ a meromorphic function, the homology with local
coefficients $H_\bullet(\mathbb{C}^n\setminus\lambda_a^{-1}(\{0,\infty\});\mathbf{L}(q))$ can be
defined using the complex generated by singular chains with coefficients in
$\mathbb{Z}[q,q^{-1}]$ on the infinite cover
$X$ of $\mathbb{C}^n\setminus\lambda_a^{-1}(\{0,\infty\})$ where $\log(\lambda(x,a))$ is defined; multiplication by $q$ is
defined by the covering transformation moving each point up one sheet
(i.e. the deck transformation associated with a circuit around
$\lambda_a^{-1}(0)$). Notice that, by assigning a specific complex value
$q=\re^{2\pi \mathrm{i} z}$, we obtain the homology with a local system of
coefficients described by the function $\lambda^z$. This, in turn, is defined
using the chain groups generated over $\mathbb{Z}(\re^{2\pi \mathrm{i} z})$ by
pairs $(\phi,s)$ where $\phi$ is a singular simplex in
$\mathbb{C}^n\setminus\lambda_a^{-1}(\{0,\infty\})$ and $s$ is a specific branch of
$\lambda^z|_\phi$, quotiented by the relation $\re^{2\pi \mathrm{i}
  z}(\phi,s)\sim(\phi,\re^{2\pi \mathrm{i} z}s)$. Then, the usual boundary
operator gives a complex and its homology is denoted by
$H_*(\mathbb{C}^n\setminus\lambda_a^{-1}(\{0,\infty\});\mathbb{Z}(\re^{2\pi \mathrm{i} z}))$.

The homology groups $H_n(\mathbb{C}^n\setminus\lambda_a^{-1}(\{0,\infty\}),\lambda_a^{-1}(0);\mathbf{L}(q))$ of $\mathbb{C}^n\setminus\lambda_a^{-1}(\{0,\infty\})$ relative to (a tubular neighbourhood of) $\lambda_a^{-1}(0)$  with
local coefficients can be defined along the same lines as for the absolute case.

\begin{rmk}
It is interesting to notice \cite{MR936695, MR2070050} that, considering the suspension $\tilde{\lambda}_a:=\lambda_a(x)-y^2$ and the zero sets $V_a=\lambda_a^{-1}(0)$, $\tilde{V}_a=\tilde{\lambda}_a^{-1}(0)$
\beq\nn 
H_n(\tilde{V}_a)\simeq H_n(\mathbb{C}^n\setminus\lambda_a^{-1}(\{0,\infty\}),\lambda_a^{-1}(0);\mathbb{Z}(-1)),
\eeq
\beq\nn
H_n(V_a)\simeq
H_n(\mathbb{C}^n\setminus\lambda_a^{-1}(\{0,\infty\}),\lambda_a^{-1}(0);\mathbb{Z}(1)).
\eeq
In this way we could say that in twisted Picard-Lefschetz theory the study of 
$\widehat{\lambda}$ interpolates between a superpotential $\lambda$ and its
suspension $\tilde{\lambda}$. The relevance of suspensions in local mirror
symmetry has already been pointed out in various places in the literature, see {\it e.g.}
\cite{Hori:2000ck, MR2651908}.
\end{rmk}

\subsection{Twisted periods}
Let us now turn to the computation of the loop integrals \eqref{eq:loopint}. A basis $\{\gamma_1,\gamma_2\}$ for
$H_1(\mathbb{P}^1\setminus\{0,\re^v,\re^{v+w},\infty\},\{0,\re^v\};\mathbf{L}(q))$
is given by any lift to $X$ of the two relative paths on
$(\mathbb{P}^1\setminus\{0,\re^v,\re^{v+w},\infty\},\{0,\re^v\})$ issuing from $\re^v$
and encircling $\re^{v+w}$ or $\infty$ respectively (see Fig. \ref{circles}).\\
\begin{figure}[h]
\begin{center}
\includegraphics[width=12cm]{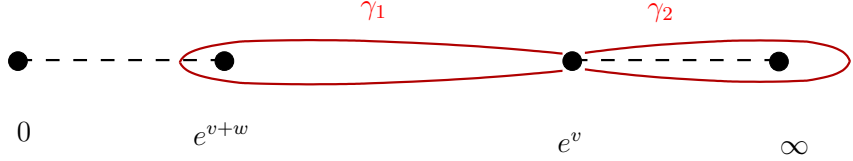}
\end{center}
\caption{Paths of integration for the twisted periods.}
\label{circles}
\end{figure}

The integration can be performed explicitly by making $\pi(\gamma_1)$ and $\pi(\gamma_2)$ tend to the segments $[\re^{v+w}, \re^{v}]$ and $[\re^{v},\infty]$. Indeed it is easy to see that
\beq
\int_{C_\epsilon^{\exp(v+w)}} \re^{z \widehat{\lambda}(p)} \frac{\di p}{p} \,
\to\ 0\ ,\ \int_{C_\epsilon^{\exp(v)}} \re^{z \widehat{\lambda}(p)} \frac{\di
  p}{p} \, \to\, 0\ ,\ \int_{C_\epsilon^{\infty}} \re^{z \widehat{\lambda}(p)}
\frac{\di p}{p}\, \to\, 0
\eeq
as $\epsilon \to 0$, where $C_{\epsilon}^a$ is the (non-closed) lift of a circle of radius $\epsilon$ around $p=a$ and $-1<z<0$.\\

Moreover, using Euler's integral representation for the hypergeometric function
\begin{equation}\label{euler}
\phantom{.}_2F_1(a,b,c,z)=\frac{\Gamma(c)}{\Gamma(b)\Gamma(c-b)}\, \int_0^1 x^{b-1} (1-x)^{c-b-1} (1-zx)^{-a} \di x,
\end{equation}
(for $\mathfrak{Re}(c)>\mathfrak{Re}(b)>0$) we can express the remaining line integrals as
\beq
\begin{split}
\mathfrak{p}_1(z)=&\frac{1}{2\pi \ri
}\left(\int_{[\re^{v+w},\re^{v}]}^{+}+\int_{[\re^{v},\re^{v+w}]}^-\right)\lambda^z(p)\frac{\di
  p}{p}= \frac{1}{2\pi \ri } (1-\re^{2 \pi \ii z})\, \int_{\re^{v+w}}^{\re^v}\,\lambda^z(p) \frac{\di p}{p}=\\
&= -  z \re^{2\pi \ri z} \re^{zv} (1-\re^w) \phantom{.}_2F_1(1-z,1+z,2,1-\re^w)
\end{split}
\label{eq:p1}
\eeq
where we used $\Gamma(1+z)\Gamma(1-z)=\frac{\pi z}{\sin(\pi z)}$ and we applied the change of variables $p=\re^v(1-(1-\re^w)x)$ in (\ref{euler}).\\

Similarly,
\begin{equation}
\begin{split}
\mathfrak{p}_2(z)=&\frac{1}{2\pi \ri }\left(\int_{[\re^v,\infty]}^{+}+\int_{[\infty,\re^v]}^-\right)\lambda^z(p)\frac{\di p}{p}= \frac{1}{2\pi \ri }(1-\re^{2 \pi \ii z})\, \int_{\exp(v)}^{\infty}\,\lambda^z(p) \frac{\di p}{p}=\\
&=  \re^{\pi \ri z} \re^{zv} \phantom{.}_2F_1(z,-z,1,\re^w)
\end{split}
\label{eq:p2}
\end{equation}
where, this time, $x=p$ in (\ref{euler}).\\
\begin{rmk}
It is worthwhile to stress what happens when we specialize to the suspension
by setting $z=1/2$, that is, when we compute the {\it odd periods} of
$\cM_{AL}$. In this case, the universal cover $X$ reduces to an elliptic
curve, and we obtain 
\bea
\label{eq:sw1}
\mathfrak{p}_1\l(\frac{1}{2}\r) &=& \frac{2 e^{v/2}
   \left(E\left(1-e^w\right)-K\left(1-e^w\right)\right)}{\pi 
   \left(e^w-1\right)} \\
\mathfrak{p}_2\l(\frac{1}{2}\r) &=& -\frac{2 i e^{v/2} E(e^w)}{\pi }
\label{eq:sw2}
\eea
where $K(x)$ and $E(x)$ denote the complete elliptic integrals of the first
and second kind respectively. Remarkably, upon identifying $u:=\re^v(1-\re^{w/2})$, $\Lambda^4 := \re^{v+w}$,
\eqref{eq:sw1}-\eqref{eq:sw2} yield respectively the derivative of the
effective prepotential and the quantum Coulomb branch parameter of $\cN=2$
$SU(2)$ super Yang--Mills theory in four dimensions \cite{Seiberg:1994rs}, in
a vacuum parameterized by a classical value $u=\bra \tr \phi^2 \ket$ for the
adjoint scalar and with RG invariant scale $\Lambda$.
\end{rmk}

The twisted periods \eqref{eq:p1}-\eqref{eq:p2} give a basis of solutions for
the deformed flatness conditions associated with the prepotential
\eqref{eq:f0inhom}.  We thus recover from the LG perspective the observation of \cite{Brini:2010ap}
that the quantum differential equation for
$\cO_{\bbP^1}(-1)_{[-\nu]}\oplus \cO_{\bbP^1}(-1)_{[\nu]}$ factorizes in the
product of an exponential ODE satisfied by $v$ (as required by the string
-axiom) and a Gauss ODE in the variable $w$. 
The choice of cycles in \eqref{eq:p1}-\eqref{eq:p2} however turns out to be non-canonical from
the point of view of Gromov--Witten theory; in particular, they are related by
an affine transformation to the \emph{topological} deformed flat coordinates $\mathfrak{p}_\alpha^{\mathrm{top}}(z)=\sum_{p=0}^\infty \mathfrak{p}_{\alpha,p}^{\mathrm{top}}$ for the resolved conifold, i.e. those deformed flat coordinates for the associated non-homogeneous Frobenius manifold such that
\beq
\mathfrak{p}_{\alpha,p}^{\mathrm{top}}:=\partial_{t^{\alpha,p}} \cF_0,
\eeq
where now $\cF_0$ is the restriction to genus zero and to primary fields of the  Gromov--Witten
potential \eqref{eq:potX} of $\cO_{\bbP^1}(-1)_{[\nu]}\oplus
\cO_{\bbP^1}(-1)_{[-\nu]}$. Explicitly, it was found in \cite{Brini:2010ap} that
\beq
\begin{split}
\mathfrak{p}^{\rm top}_1(z)=&
\l(-\frac{1}{z}+\pi  \cot (\pi  z)-2 \gamma\r) \phantom{.}_2F_1(-z,z,1,\re^w) \re^{vz}\\
&- \frac{\pi z}{\sin{\pi z}} \phantom{.}_2F_1(z+1,-z+1,2,1-\re^w)(1-\re^w)\re^{vz}\\
\mathfrak{p}^{\rm top}_2(z)=&\frac{\phantom{.}_2F_1(-z,z,1,\re^w)\re^{vz}-1}{z}
\end{split}
\eeq
where $\gamma$ is the Euler-Mascheroni constant. By comparing with our formulas above for the oscillating integrals, we find
$$\left(\begin{array}{c} \mathfrak{p}^{\mathrm{top}}_1(z)\\ \mathfrak{p}_2^{\mathrm{top}}(z)\end{array}\right)\,=\, \left(
\begin{array}{cc}
0 & \frac{\re^{-\mathrm{i} \pi z}}{z} \\
\frac{\re^{-2 \ri \pi  z} (1+2 z \gamma-\pi  z \cot (\pi  z))}{z^2}
 & -\pi  \re^{-\ri \pi  z} z \csc (\pi  z) 
\end{array}
\right) \left(\begin{array}{c} \mathfrak{p}_1(z)\\ \mathfrak{p}_2(z)\end{array}\right)-\left(\begin{array}{c} \frac{1}{z}\\ 0\end{array}\right).$$

\section{Outlook}
\label{sec:concl}
We list here some possible developments of this work. 
First of all, the bi-Hamiltonian structure we constructed deserves further investigation: it would be important, on one hand, to elucidate the relation with the bi-Hamiltonian structure of 2D-Toda, and on the other, to find a full dispersive formulation of the pencil. Moreover, the factorization of the Lax matrices that we observe in the Toeplitz reduction points to a generalization of this system to a class of rational reductions of the 2D-Toda hierarchy and, in the dispersionless limit, to corresponding examples of (almost) Frobenius manifolds.

Secondarily, our study of the Frobenius structure \eqref{eq:f0hom}
suggests that an interesting generalization of the Dubrovin--Zhang theory
should find a place in the case of conformal Frobenius manifolds with non-costant unit,
and, correspondingly, of bi-Hamiltonian hierarchies with non-exact Poisson
pencils. On a more practical note, we remark also that the second half of the Levelt basis for
\eqref{eq:f0hom} needs further understanding and an explicit construction. 

On the dual side, an enticing possibility would be to leverage our
construction of a Landau--Ginzburg mirror in order to shed some light on the
higher genus theory, and to generalize the picture to more general target
spaces. We leave these problems for future investigation.

\begin{appendix}
\section{The semi-infinite Toeplitz lattice}
\label{sec:semiinf}
The semi-infinite Lax matrices are obtained by restricting the matrix indices in  ~\eqref{L12-1}-\eqref{L12-2}  to non-negative values.
They are indeed simply given by the lower-right blocks in~\eqref{L12-m}:
\beq
L_1 = \left(\begin{tabular}{lllll}
$-x_1y_0$  &  $1$ &  $0$      &  $0$ &   \\
$-v_1x_2y_0$ &  $-x_2y_1$ & $1$&  $0$   & \\ 
$-v_1 v_2 x_3y_0$ &  $-v_2 x_3y_1$ & $ -x_3y_2$&  $1$ & \\
$ -v_1 v_2 v_3 x_4y_0$ &  $ -v_2 v_3 x_4y_1$ & $-v_3 x_4y_2$  & $ -x_4y_3$   &
\\
 & &  &    &  $\ddots$\\
\end{tabular}
\right), \nn
\eeq
\beq
L_2 =
\left(\begin{tabular}{lllll}
$-x_0y_1$  &  $-x_0y_2$ & $-x_0y_3$     & $-x_0y_4$ &   \\
$1 -x_1y_1$ &  $-x_1y_2$  & $-x_1y_3$& $-x_1y_4$   & \\
~~$0$       &  $1 -x_2y_2$ & $ -x_2y_3$&  $-x_2y_4$ & \\
~~$0$       &  ~~$0$      & $ 1 -x_3y_3$  & $ -x_3y_4$   &  \\
 & &  &    &  $\ddots$\\
\end{tabular}
\right). \nn
\eeq

The first Lax matrix still factorizes as 
\[
L_1 = A B^{-1} 
\]
with semi-infinite matrices $A$, $B$ still given by~\eqref{aby}, while for $L_2$ we have
\beq \label{hatl}
L_2 = B A^{-1} + E
\eeq
where the matrix $E$, which is zero except for the first row, is given by
\[
E_{n,m} = \frac{v_0}{y_0} \delta_{n,0} y_{m+1}, \qquad n,m\geq0.
\]
This is due to the fact that, in the semi-infinite case, the matrix $\L^{-1}$ is the right-inverse of $\L$ but not its left-inverse; indeed, in this case
\beq \label{leftinv}
\L^{-1} \L = 1 - \mathcal{E}
\eeq
with $\mathcal{E}$ everywhere zero but in the upper-left corner, where it is equal to $1$. 
In the proof of Proposition~\ref{prop-fact} the left-inverse property of $\L^{-1}$ is used only in the factorization of $L_2$; using~\eqref{leftinv} instead, we easily obtain~\eqref{hatl}.

In the semi-infinite case, the matrix product $L_1 L_2$ does not involve infinite sums and the problem with associativity mentioned in Remark~\ref{assoc} does not arise. We need however to correct $L_2$ with the contribution of $E$; we have
\begin{prop}
In the semi-infinite Toeplitz lattice the Lax matrices satisfy the constraint
\[
L_1 \hat L_2 =1,
\]
where 
\beq
\hat L_2 = L_2 - E = 
\left(\begin{tabular}{lllll}
$-\frac{y_1}{y_0}$  &  $-\frac{y_2}{y_0}$ & $-\frac{y_3}{y_0}$     & $-\frac{y_4}{y_0}$ &   \\
$1 -x_1y_1$ &  $-x_1y_2$  & $-x_1y_3$& $-x_1y_4$   & \\
~~$0$       &  $1 -x_2y_2$ & $ -x_2y_3$&  $-x_2y_4$ & \\
~~$0$       &  ~~$0$      & $ 1 -x_3y_3$  & $ -x_3y_4$   &  \\
 & &  &    &  $\ddots$\\
\end{tabular}
\right). \nn
\eeq
\end{prop}
Note that the associativity problem is still present for the product $\hat L_2 L_1$, which is not equal to the identity matrix. 

A special role is played by the constraint $x_0 y_0 =1$, which is in particular satisfied by the solution of the semi-infinite Toeplitz lattice obtained from the unitary matrix model~\cite{MR1794352}. Under such constraint, which is clearly preserved by the hierarchy (cf. equation~\eqref{eq:lateom}), the matrix $E$ vanishes, hence $L_1 L_2 =1$. 

\begin{rem}
Adler - van Moerbeke have shown~\cite{MR1794352, MR1993333} that the ratios
\[
x_n = (-1)^n \frac{\tau_n^{(1)}}{\tau_n^{(0)}}, \quad
y_n = (-1)^n \frac{\tau_n^{(-1)}}{\tau_n^{(0)}}, \quad
n \geq1
\]
of the tau-functions of the unitary matrix integral defined by
\[
\tau_n^{(k)} = \int_{U(n)} (\det M)^k \re^{\sum_{j=1}^{\infty} \tr (s^{(1)}_j M^j - s^{(2)}_j \bar M^j )} \ dM 
\]
satisfy the semi-infinite Toeplitz lattice with $x_0(s)=y_0(s)=1$ and $x_n(0)=y_n(0)=0$ for $n>0$. Since $v_0=0$, for such solution we have $L_1 L_2 =1$.
\end{rem}

\end{appendix}
\bibliography{miabiblio}
\bibliographystyle{amsalpha}

\end{document}